\documentclass{article}

\usepackage[latin1]{inputenc}
\usepackage{amsfonts, amsmath, amsthm, amssymb}
\usepackage{color, graphicx, caption, booktabs,minibox}

\usepackage{lscape}

\newtheorem{theorem}{Theorem}[section]
\newtheorem{corollary}[theorem]{Corollary}

\newtheorem{lemma}[theorem]{Lemma}

\newcommand{\ds}{\displaystyle}

\def\cS{\mathcal{S}}
\def\cA{\mathcal{A}}
\def\cB{\mathcal{B}}
\def\cQ{\mathcal{Q}}
\def\cR{\mathcal{R}}

\def\cM{\mathcal{M}}
\def\cN{\mathcal{N}}
\def\cN{\mathcal{N}}
\def\cD{\mathcal{D}}
\def\cP{\mathcal{P}}
\def\cL{\mathcal{L}}
\def\cH{\mathcal{H}}
\def\cF{\mathcal{F}}

\title{Enumeration of labelled 4-regular planar graphs}

\author{
	Marc Noy 
    \thanks{Universitat Polit\`ecnica de Catalunya and Barcelona Graduate School of Mathematics, Department of Mathematics, Edifici Omega, 08034 Barcelona, Spain. 
    E-mail: {\tt marc.noy@upc.edu}. 
    Supported by the Spanish Ministerio de Econom\'{i}a y Competitividad projects MTM2014-54745-P, MDM-2014-0445 and MTM2017-82166-P.}
\and
	Cl\'ement Requil\'e	
	\thanks{
		Technische Universit\"at Wien, Institute of Discrete Mathematics and Geometry, Wiedner Hauptstrasse 8, 1040 Wien, Austria. 
		E-mail: {\tt clement.requile@tuwien.ac.at}. 
		Supported by the Special Research Program F50 \textit{Algorithmic and Enumerative Combinatorics} of the Austrian Science Fund.}
\and
    Juanjo Ru\'e 
    \thanks{Universitat Polit\`ecnica de Catalunya and Barcelona Graduate School of Mathematics, Department of Mathematics, Edifici Omega, 08034 Barcelona, Spain. 
    E-mail: {\tt juan.jose.rue@upc.edu}. 
    Supported by the Spanish Ministerio de Econom\'{i}a y Competitividad projects MTM2014-54745-P, MDM-2014-0445 and MTM2017-82166-P.}
}

\begin{document}

\maketitle

\begin{abstract}
We present the first combinatorial scheme for counting labelled 4-regular planar graphs through a complete recursive decomposition. More precisely, we show that the exponential generating function of labelled 4-regular planar graphs can be computed effectively as the solution of a system of equations, from which the  coefficients can be extracted.
As a byproduct, we also enumerate labelled 3-connected 4-regular planar graphs, and simple 4-regular rooted maps.
\end{abstract}

\section{Introduction}

The enumeration of labelled planar graphs  has been recently  the subject of much research; see \cite{ICM, handbook} for  surveys on the area.
The problem of counting planar graphs was first solved by Gim\'enez and Noy \cite{gn}, while cubic planar graphs where enumerated in Bodirsky, Kang, L\"offler and McDiarmid \cite{bklm2007} (see also \cite{cubic-revisited} for an update). On the other hand, the enumeration of simpler classes of planar graphs, such as series-parallel graphs and, more generally,  subcritical classes of graphs is easier and well understood~\cite{subcritical}.

One of the open problems in this area is the enumeration of labelled \emph{4-regular} planar graphs.
There are several references on the exhaustive \emph{generation} of 4-regular planar graphs.  Starting with a collection of basic graphs one shows how to generate all graphs in a certain class starting from the basic pieces and applying a sequence of local modifications.
This was first done for the class of 4-regular planar graphs by Lehel \cite{Lehel}, using as basis the graph of the octahedron. For 3-connected 4-regular planar graphs a similar generation scheme was shown by Boersma, Duijvestijn and  G\"obel
 \cite{BDG93}; by removing isomorphic duplicates they were able to compute the numbers of 3-connected 4-regular planar graphs up to 15 vertices. It  is also the approach of the more recent work by Brinkmann, Greenberg, Greenhill, McKay, Thomas and Wollan \cite{BGGMTW05} for generating planar quadrangulations of several types. The authors of \cite{BGGMTW05} use several enumerative formulas to check the correctness of their generation procedure. However this does not include the class of
3-connected quadrangulations, which by duality correspond to 3-connected 4-regular planar graphs, a class for which no enumeration scheme was known until now.

In the present  paper we provide the first scheme for counting 4-regular planar graphs through a complete recursive decomposition. In what follows all graphs are labelled.
Let $C(x)= \sum c_n \frac{x^n}{n!}$ be the exponential generating function of labelled connected 4-regular planar graphs counted according to the number of vertices. We show that the derivative $C'(x)$ can be computed effectively as the solution of a (rather involved) system of  algebraic equations.
In particular $C'(x)$ is an algebraic function.
Using a computer algebra system we can extract the coefficients of $C'(x)$, hence also of $C(x)$.  If $G(x) = \sum g_n \frac{x^n}{n!}$ is now the generating function of all  4-regular planar graphs, the exponential formula $G(x) = e^{C(x)}$ allows us to find the coefficients $g_n$ as well.

Our main result is the following:
\begin{theorem}\label{th:main}
The generating function $C'(x)$ is algebraic and is expressible as the  solution of a system of algebraic equations from which one can effectively compute their coefficients.
\end{theorem}

As a corollary we obtain:
\begin{corollary}\label{coro}
The sequence  $\{g_n\}_{n\ge0}$ is $P$-recursive, that is, it satisfies a linear recurrence with polynomial coefficients.
\end{corollary}

We also determine the generating function $\tau(x) = \sum t_n \frac{x^n}{n!}$ of 3-connected 4-regular planar graphs and show that its derivative is algebraic.

\begin{theorem}\label{th:3-conn}
The generating function $\tau'(x)$ is algebraic and is expressible as the  solution of a system of algebraic equations from which one can effectively compute their coefficients.
\end{theorem}

\medskip
To obtain our results we follow the classical technique introduced by Tutte: take a graph rooted at a directed edge and classify the possible configurations arising from the removal of the root edge.
This produces several combinatorial classes that are further decomposed, typically in a recursive way.
The combinatorial decomposition translates into a system of equations for the associated generating functions, which in our case is considerably involved.
Next we provide a brief overview of the combinatorial scheme in our solution.

Using a variant of the classical decomposition of 2-connected graphs into 3-connected components in the spirit of \cite{bklm2007}, we find an equation linking $C(x)$ to the generating function $T(u,v)$ of 3-connected 4-regular planar graphs, counted according to the number of simple edges and the number of double edges.
Actually, $T$ will be the generating function of rooted 3-connected \emph{maps} (a rooted map is an embedding of a planar graph where a directed edge is distinguished) but by Whitney's theorem, 3-connected planar graphs have a unique embedding in the oriented sphere and in this situation counting graphs is equivalent to counting maps.
As will be seen later, it is essential to count 3-connected maps according to simple and double edges, otherwise there is not enough information to obatin $C(x)$.

Once we have access to $T(u,v)$ we can compute the coefficients of $C(x)$ to any order.
In order to compute $T(u,v)$ we apply the reverse procedure working with maps instead of graphs.
The starting point is the fact that the number $M_n$ of rooted 4-regular maps with $2n$ edges is well known, since they are in bijection with rooted (arbitrary) maps on $n$ edges, and equal to (see Formula (5.1) in \cite{tutte})
$$
	M_n = \frac{2\cdot 3^n}{(n+1)(n+2)}\binom{2n}{n}.
$$
Using again the decomposition into 3-connected components, one can obtain an equation linking $T(u,v)$ and $M(z)  = \sum M_n z^n$.
However this is not sufficient since $T(u,v)$ is a bivariate series and cannot be recovered uniquely from the univariate series $M(z)$.
In order to overcome this situation, we enrich the combinatorial scheme and count maps according to a secondary parameter: the number of isolated faces of degree 2, namely, those not incident with another face of degree 2.
Notice that this parameter, when restricted to 3-connected 4-regular planar graphs, is precisely the number of double edges.

If $M(z,w)$ is the associated series, where $w$ marks the new parameter, then we can enrich the corresponding equations and obtain an algebraic relation of the form
$$
	T(f(z,w,M(z,w)),g(z,w,M(z,w))) = h(z,w,M(z,w)),
$$
where $f,g$ and $h$ are explicit functions.
The transformation
$$
	u=f(z,w,M), \qquad v=g(z,w,M)
$$
turns out to have non-zero Jacobian and can be inverted explicitly.
This allows us to express $T(u,v)$ as a power series whose coefficients can be computed in terms of those of $M(z,w)$ and the inverse mapping $(z,w)\to (u,v)$. From here, we can compute the coefficients of $T(u,v)$ to any order.
In particular, the coefficients of $T(u,0)$ give the numbers $T_n$ of simple rooted  3-connected 4-regular planar maps.
By double counting, we obtain the number of labelled 3-connected 4-regular planar graphs as $t_n = T_n (n-1)!/8$.

It remains to compute $M(z,w)$.
To this end, we use the dual bijection between 4-regular maps and quadrangulations, and count quadrangulations according to the number of faces and the number of vertices of degree 2 not adjacent to another vertex of degree 2.
This is technically demanding but can be achieved using the decomposition of quadrangulations along faces and edges, refined to take into account the new parameter.

Once we have access to $T(u,v)$, we can also enumerate \emph{simple} 4-regular maps, a result of independent interest (the enumeration of simple 3-regular maps can be found in  \cite{cubicMaps}).

\begin{theorem}\label{th:maps}
The generating function of rooted simple 4-regular  maps is algebraic, and is expressible as the solution of a system of algebraic equations from which one can compute  its coefficients effectively.
\end{theorem}

The  steps   towards the proofs of Theorems \ref{th:main} and \ref{th:maps} are summarized in the following diagram.

\medskip

\begin{center}
\begin{tabular}{c}
	\minibox[frame]{Simple quadrangulations  (Section 2.1)} \\
	$\big\Downarrow$ \\
	\minibox[frame]{Arbitrary quadrangulations  $\longleftrightarrow$  4-regular maps (Section 2.2)} \\
		$\big\Downarrow$ \\
 \minibox[frame]{3-connected  4-regular graphs
	 $\longleftrightarrow$ 3-connected  4-regular maps  (Section 3)} \\
	$\big\Downarrow$ \hskip5cm $\big\Downarrow$  \\
\minibox[frame]{ 4-regular graphs  (Section 4)}  \qquad \minibox[frame]{Simple 4-regular maps (Section 5)}
	\end{tabular}
\end{center}

\bigskip

We remark that the final equations relating the power series $M(z,w)$, $T(u,v)$ and $C(x)$ are  not written down explicitly.
Instead, we work with several intermediate systems of equations that allow us to extract the coefficients of the corresponding series.
This is computationally demanding but it is within the capabilities of a computer algebra system such as \texttt{Maple}.

Here is a summary of the paper.
In Section \ref{sec:quad}, we determine the series $M(z,w)$ as the solution of a system of polynomial equations.
In Section \ref{sec:3conn}, we  obtain $T(u,v)$ as a computable function of $M(z,w)$, thus proving the second part of Theorem~\ref{th:main}.
In Section \ref{sec:graphs}, we find an equation connecting $C(x)$ and $T(u,v)$, which allows us to compute the coefficients of $C(x)$, that is, the number of connected 4-regular planar graphs, and to complete the proof of Theorem \ref{th:main}.
From the relation $G(x) = \exp(C(x))$, we obtain the coefficients of $G(x)$. Finally, in Section \ref{sec:simple-maps} we count simple 4-regular maps.

\section{Counting quadrangulations}\label{sec:quad}

All maps in this paper  are rooted; this means that an edge $uv$ is marked and directed from $u$ to $v$.
Then $u$ is the root vertex and the face to the right of $uv$ is the root face, which by convention is taken to be the outer face.
A rooted map has no non-trivial automorphism, hence all vertices, edges and faces are distinguishable.

A quadrangulation is a planar map in which all faces have degree 4.
{Vertices in the outer face are called \emph{external} vertices.}
{A \emph{diagonal} in a quadrangulation is a path of length 2 joining opposite vertices of the external face and whose middle vertex is not external.}
If $uv$ is the root edge, there are two kind of diagonals, those incident with $u$ and those incident with $v$.
By planarity not both types can be present at the same time.
A cycle of length 4 which is not the boundary of a face is called a \emph{separating quadrangle}.
A vertex of degree 2 is \emph{isolated} if it is not adjacent to another vertex of degree 2.
An isolated vertex of degree 2 is called a \emph{2-vertex}.

\subsection{Simple quadrangulations}

All quadrangulations in this section  are simple, that is, have no multiple edges. This implies in particular that all faces are quadrangles, that is, simple 4-cycles.
Following \cite{ms68}, we consider the following classes of quadrangulations, illustrated in Figure \ref{fig:simple_quad}:

\begin{itemize}
    \item
        $\cQ$ are all (simple) quadrangulations.

    \item
        $\cS$  are quadrangulations {with at least 8 vertices and} without diagonals or separating quadrangles.

    \item
        $\cN$ are quadrangulations containing a diagonal incident with the root vertex. By symmetry they are in bijection with quadrangulations containing a diagonal not incident with the root vertex.

    \item
        $\cN_i$ are quadrangulations in $\cN$ with exactly $i$ external 2-vertices, for $i = 0,1,2$.

    \item
        $\cR$  are quadrangulations obtained from $\cS$ by possibly replacing each internal face with a quadrangulation in $\cQ$.
\end{itemize}

\noindent{Observe that the map consisting of a single quadrangle is in $\cQ$, but is not  in any of the other classes.}

\begin{figure}[htb]
\centering
	\includegraphics[scale=1.]{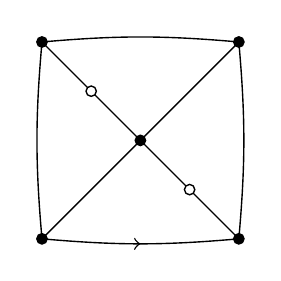}
	\includegraphics[scale=1.]{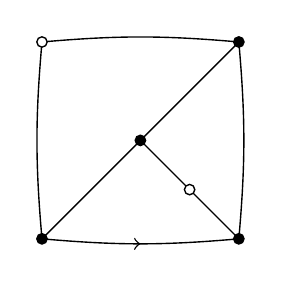}
	\includegraphics[scale=1.]{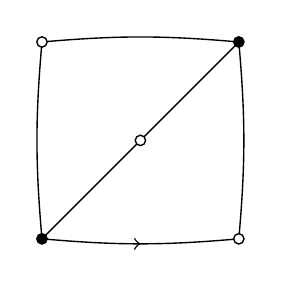}
	\includegraphics[scale=1.]{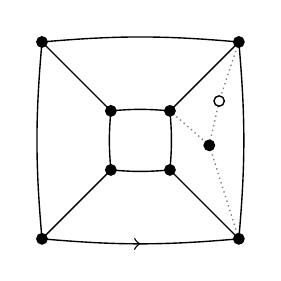}
	\caption{\em  Examples of simple quadrangulations.
		The first three, from left to right, belong to $\cN_0$, $\cN_1$ and $\cN_2, respectively$.
 		The last one is in $\cR$: a quadrangulation from $\cS$ in which one face has been replaced with a quadrangulation in $\cN_1$.
 		Isolated vertices of degree 2 are shown in white.}
	\label{fig:simple_quad}
\end{figure}

\noindent
In the following generating functions, $z$ marks internal faces and $w$ marks 2-vertices.
For each class of quadrangulations we have the corresponding generating function written with the same letter.
For instance $Q(z,w)$ is associated with the class $\cQ$, and so on.
The exception is $S(z)$, since a quadrangulation in $\cS$ has no vertex of degree 2, and variable $w$ does not appear.

The generating function $S(z)$ is well-known, since $\cS$ is in bijection with rooted 3-connected planar maps counted by the number of edges.
The generating function of 3-connected maps according to vertices and edges is given in Equation (9) from \cite{BGW2002} as $M_3(x,z)$.
Setting $x = 1$, we obtain that $S(z)$ is given by
$$
	S(z) = \ds\frac{2z}{1+z} -z - \frac{U(z)^2}{z(1+2U(z))^3}, \qquad
	U(z) = z(1+U(z))^2.
$$
The smallest term in $S(z)$ is $z^5$, corresponding to the graph of the cube.

The following lemma guarantees the existence of a unique non-zero solution to a system of non-negative equations.

\begin{lemma}\label{positive}
	Let $y_1(z,u), \dots, y_m(z,u)$ be power series satisfying the system of equations
	$$
		\begin{array}{lcl}
			y_1 &=& F_1(z, y_1, \dots , y_m, u), \\
			y_2 &=& F_2(z, y_1, \dots , y_m, u),\\
			&\vdots& \\
			y_m &=& F_m(z, y_1, \dots , y_m, u),\\
		\end{array}
	$$
	where the $F_i $ are power series in the variables indicated.

	Assume that for each $i$, $F_i$ has non-negative  coefficients and is divisible by~$z$.
	Assume also that there exists a solution $\mathbf{\widehat{y}}=(y_1(z,u),\dots,y_m(z,u))$ to the system which is not identically $0$ for all $i$.
	Then this is the unique  solution with non-negative coefficients.
	
	Moreover, the solution can be computed iteratively from the initial values $y_i=0$ up to any degree of $z$.
\end{lemma}

\begin{proof}
	We first recall that the order of a non-zero power series $A(z)=\sum a_n z^n$ is the minimum $n$ such $a_n \ne 0$, and that a sequence $\{A_m(z)\}_{m\ge0}$ is convergent in the ring of formal power series if the order of the $A_m(z)$ go to infinity.

	Let $\mathbf{F} = (F_1,\dots, F_m)$.
	Start with the initial value $\mathbf{y}^{(0)}(z,u)=0$ and let $\mathbf{y}^{(k+1)}(z,u)= \mathbf{F}(z,\mathbf{y}^{(k)}(z,u),u)$, where $\mathbf{y}^{(k)}(z,u) = (y_1^{(k)}(z,u), \dots, y_m^{(k)}(z,u))$.
	Since the $F_i$ have non-negative coefficients, so do the $\mathbf{y}^{(k)}$.
	Since each $F_i$ is divisible by $z$, the mapping $\mathbf{F}$ is a contraction, in the sense that the order of $y_i^{(k+1)}(z,u)$ is larger than the order of  $y_i^{(k)}(z,u)$.
	Hence the solution is unique and is given by the limit of the $\mathbf{y}^{(k)}(z,u)$. The solution is non-zero since the $F_i$ are non-zero.
\end{proof}

The former proof gives a procedure for computing iteratively the unique solution.
Start with $y_i=0$ for all $i$ and compute the $y_i^{(k)}$ iteratively.
Each ${y}_i^{(k)}(z,u)$ is a polynomial and, because of the hypothesis on the $F_i$, ${y}_{i+1}^{(k+1)}(z,u)= {y}_i^{(k)}(z,u) + M_i$, where $M_i$ is a monomial of degree larger than the degree of ${y}_i^{(k)}(z,u)$.
This can be iterated to any desired degree.
Observe also that the variable $u$ plays only the role of a parameter and that further parameters can be added without any change.

For the various systems of equations that we encounter it is always easy to check that the conditions of the previous result are satisfied.
The existence of a non-zero solution with non-negative coefficients follows from the combinatorial power series involved in the system, like in the following lemma.

\begin{lemma}\label{lem:Q}
	Let
	$$
		N = N_0 + N_1 + N_2, \qquad \widetilde{N} = N_0 + \frac{N_1}{w} + \frac{N_2}{w^2}.
	$$
	Then the following system of equations holds:
	\begin{equation}\label{eqsQ}
		\renewcommand\arraystretch{1.5}
		\begin{array}{ll}
	 		Q =& z+2N+R, \\
		 	R =& S(z+2\widetilde{N}+R), \\
	 		N_0 =& (\widetilde{N} + R)\left(\widetilde{N} + R + N_0 + \ds\frac{N_1}{2w}\right), \\
 	  		N_1 =& 2zw\left(\widetilde{N} + R + N_0 + \ds\frac{N_1}{2} \right), \\
	 		N_2 =& z^2w^3 + zw\left(\ds\frac{N_1}{2}+ N_2 \right).
		\end{array}
	\end{equation}
	Moreover, the system has a unique non-zero solution with non-negative coefficients.
\end{lemma}

\begin{proof}
	First we check that the system has non-negative coefficients.
As $S(z)$ has non-negative Taylor coefficients, we only need to argue on the terms $N_1/w$ and $N_2/w^2$.
But by definition, quadrangulations in $\cN_i$ have at least  $i$ 2-vertices, and the generating function $N_i$ has $w^i$ as a factor.
It is immediate to check that all the right-hand terms are divisible by $z$, hence Lemma \ref{positive}  guarantees the last claim in the statement.

The first equation follows from the fact that a quadrangulation, not reduced to a single quadrangle, either has a diagonal or is obtained from a quadrangulation in $\cS$ by replacing internal faces with arbitrary quadrangulations in $\cQ$.

The second equation expresses the recursive nature of the class $\cR$.
Notice that the substitution inside $S$ contains the term $\widetilde{N}$ instead of $N$. The reason is that vertices of degree 2 in the outer face become vertices of degree more than two after  substitution.

For the rest of the proof, notice that the left-hand terms in the last three equations correspond to  quadrangulations in $\cN$, whose diagonal is incident to the root vertex.
They are decomposed following their rightmost diagonal, that is, the first diagonal to the left of the root edge.
Let $A$ be the quadrangulation consisting of a single quadrangle with a diagonal {adjacent to the root vertex}.
It has two inner faces: $f_1$, incident with the root edge, and $f_2$.
A quadrangulation in $\cN$ is obtained by replacing the inner faces of $A$ with simple quadrangulations, such that the quadrangulation replacing $f_1$ does not have a diagonal incident with the root vertex, as otherwise
the diagonal of $A$ would not be the rightmost diagonal of the resulting quadrangulation, contrary to the construction.
In what follows, the {\em top vertex} is the external vertex adjacent to the root vertex and not incident with the root edge.

\emph{Equation for $N_0$.}
Both $f_1$ and $f_2$ can be replaced either with quadrangulations in $\cR$ or those with a diagonal not adjacent to the root vertex (which are in bijection with $\cN$), hence the factor $\widetilde{N} + R$.
In addition, $f_2$ can be replaced by a quadrangulation with a diagonal adjacent to its root vertex, but only if the top vertex is not of degree 2, that is, any quadrangulation in $\cN_0$ and half of the ones in $\cN_1$ (those in which the top vertex is not of degree 2).

{\emph{Equation for $N_1$}. We first assume  that the unique external $2$-vertex is incident with the root edge. Then $f_1$ is not replaced, and (as in the analysis of $N_0$) $f_2$ can be replaced $\widetilde{N}+R+N_0+\frac{N_1}{2}$. The only difference  with the analysis of $N_0$ is the term $N_1/2$ (instead of $\frac{N_1}{2w}$), as the $2$-vertex remains of degree $2$. Hence, in this situation we obtain  $zw\left(\widetilde{N}+R+N_0+\frac{N_1}{2}\right)$.
Finally, the case where the unique external $2$-vertex is the top vertex can be deduced in the same way by decomposing the quadrangulation using the leftmost diagonal with the respect to the root edge instead of the rightmost diagonal.}

\emph{Equation for $N_2$.}
In this case face $f_1$ is not replaced.
If neither is $f_2$, we get $A$, hence the term $z^2w^3$.
Else, $f_2$ is replaced with a quadrangulation in $\cN$ whose top vertex is of degree 2, corresponding to the term $N_2+N_1 /2$, as before.
\end{proof}

From the previous system of equations we can compute the coefficients to any order of all the series involved by iteration.
As we are going to see, a modified version of the series $Q$, $N_0$, $N_1$, $N_2$ is needed in  Section \ref{sec:arbitrary}.

\subsection{Arbitrary quadrangulations}\label{sec:arbitrary}

A quadrangulation of a 2-cycle is a rooted map in which each face is of degree 4 except the outer face which is of degree 2.
One of the two edges in the outer face is taken as the root edge, and its tail is the root vertex.
An arbitrary quadrangulation is obtained from a simple quadrangulation by replacing edges with quadrangulations of a 2-cycle.
Conversely, given an arbitrary quadrangulation, collapsing all maximal 2-cycles, one obtains a simple quadrangulation.
Notice that among the simple quadrangulations, we need to include the degenerate case consisting of a path of length 2, which we denote by $P_3$  (see Figure \ref{fig:rootface} and the corresponding caption).

In an arbitrary quadrangulation there are three possibilities for the shape of the root face: it is either a quadrangle, the result of gluing two 2-cycles through a vertex, or gluing one 2-cycle and one edge (see Figure \ref{fig:rootface}).

\begin{figure}[htb]
\centering
	\raisebox{.5\height}{\includegraphics[scale=1]{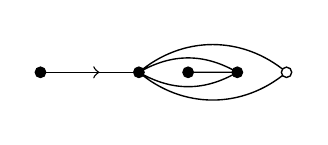}}
	\raisebox{.5\height}{\includegraphics[scale=1]{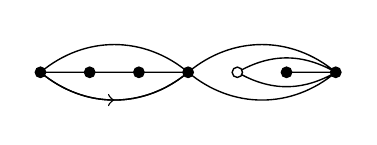}}
	\includegraphics[scale=1]{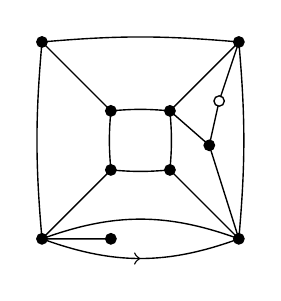}
\caption{\em The  different types of quadrangulations in $\cB$, depending on the nature of the root face.  Isolated vertices of degree 2 are shown in white.}
\label{fig:rootface}
\end{figure}

We define the following classes of quadrangulations:

\begin{itemize}
	\item
	$\cA = \cA_0 \cup \cA_1$ are quadrangulations of a 2-cycle.
	$\cA_1$ are those whose root vertex is a 2-vertex (by symmetry, they are in bijection with those in which the other  external vertex  is a 2-vertex), and $\cA_0$ are those without  external 2-vertices.
	
	\item
	$\cB= \cB_0 \cup \cB^*_0 \cup \cB_1$ are arbitrary quadrangulations.
	$\cB_1 $ are those in which the root edge is incident to exactly one 2-vertex, and $\cB_0 \cup \cB^*_0$ are those in which the root edge is not incident to a 2-vertex.
	Furthermore, $\cB_0^*$ are the quadrangulations obtained by replacing one of the two edges incident with the root edge on the single quadrangle,
	as illustrated in Figure~\ref{fig:quad_M0'_M1}.
	The class $\cB^*_0$ is introduced for technical reasons that will become apparent in the next section, when we consider the dual class of 4-regular maps.
\end{itemize}

In the generating functions $A_0(z,w)$ and $A_1(z,w)$, variables $z$ and $w$ mark, respectively, internal faces and 2-vertices, whereas in $B_0(z,w)$, $B_1(z,w)$ and $B^*_0(z,w)$ variable $z$ marks {\it all} faces: it is important to keep in mind this distinction when checking the equations satisfied by the various generating functions.
An exception, which again becomes clear when passing to the dual, is the term $2z$ encoding the path $P_3$ (which can be rooted in two different ways), where the middle vertex is not considered to be a 2-vertex.

\begin{figure}[htb]
\centering
	\begin{minipage}[c]{.55\textwidth}
		\centering
    		\includegraphics[scale=0.7]{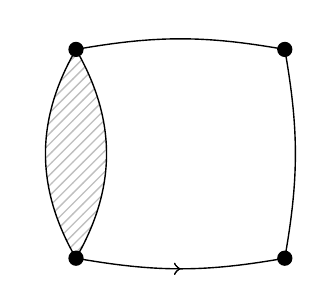}
    		\includegraphics[scale=0.7]{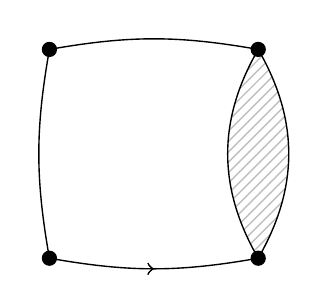}
	\end{minipage}
	\begin{minipage}[c]{.4\textwidth}
		\centering
    		\includegraphics[scale=1.2]{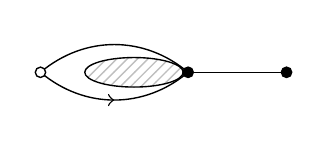}\\
    		\includegraphics[scale=1.2]{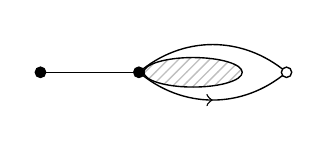}
	\end{minipage}
	\caption{\em
		On the left are the two types of quadrangulations in $\cB^*_0$.
		On the right are two types of substitutions of $P_3$ counted in $\cB_1$ .
}
	\label{fig:quad_M0'_M1}
\end{figure}

The next generating function encodes the substitution of edges in simple quadrangulations:
$$
	\widetilde{A} = A_0 + \frac{2A_1}{w}.
$$
The reason for $w$ in the denominator is that, after substitution, the 2-vertex in $\cA_1$ no longer has degree 2, and the factor 2 is because this vertex can be any of the two endpoints of the root edge.

To encode the substitution of edges by objects from $\cA$, we first remark that the number of edges in a quadrangulation is twice the number of faces, hence there is a bijection between faces and pairs of edges.
Moreover, since 2-vertices are isolated, pairs of edges incident with a 2-vertex are uniquely determined.
{In fact, one can interpret $w$ as to mark all pairs of edges incident with 2-vertices, while $z$ marks pairs of remaining edges except the pair that corresponds to the external face in the aforementioned bijection.}
These considerations justify the following change of variables:
$$	
	s= s(z,w) = (1+\widetilde{A})^2,\qquad  \quad
	t =t(z,w) = \frac{w + 2\widetilde{A} + \widetilde{A}^2}{(1+\widetilde{A})^2},
$$
whose meaning is the following.
The term $(1+\widetilde{A})^2$ in $s$ encodes pairs of edges that are possibly substituted.
Pairs of edges incident with a 2-vertex must be treated differently: if any of them is replaced with an object from $\cA$, the vertex no longer has degree 2, hence the term $w + 2\widetilde{A} + \widetilde{A}^2$ in $t$.
This must be corrected with the term $(1+\widetilde{A})^2$ in the denominator of $t$ since those edges where already counted in $s$.

Consider now the system \eqref{eqsQ} from the previous section with the change of variables
$$
	z= zs, \qquad w=t.
$$
We remark that the factor $z$ in $zs$ encodes faces in the initial simple quadrangulation.
We then partition the resulting quadrangulations into three families.
The first one is associated to the generating function $E(z,w)$ which counts those whose initial simple quadrangulation is the single quadrangle.
The second and third are counted by $Q_1(z,w)$ and $Q_0(z,w)$, which respectively count those whose root edge is incident or not with a 2-vertex, and whose initial simple quadrangulations is not the single quadrangle.
In all three generating functions, the variable $z$ marks inner faces and the variable $w$ marks 2-vertices.

\begin{lemma}\label{lemma:part}
	The following equalities hold:
	$$
 		\renewcommand\arraystretch{1.5}
		\begin{array}{lcl}
			Q_1 &=& \ds\frac{1}{t}(N_1(zs,t) + 2N_2(zs,t)), \\
			Q_0 &=& s(2N_0(zs,t) + N_1(zs,t)+ R(zs,t)) +(2\widetilde{A}+\widetilde{A}^2)Q_1, \\
			E &=& z(1+\widetilde{A})^4-4z\widetilde{A}^2  + 4zw\widetilde{A}^2,
		\end{array}
	$$
	where $z$ marks inner faces and $w$ marks 2-vertices.
\end{lemma}

\begin{proof}
Let $uv$ be the root edge.
In all three cases we consider the simple quadrangulation obtained when collapsing all 2-cycles.
The first equation holds because quadrangulations where there is one 2-vertex incident with the root edge
are encoded by $N_1 + 2N_2$, and the factor $\frac{1}{t}$  because we do not replace any of its two incident edges.

In the second equation, we need to distinguish whether either $u$ or $v$ in the initial simple quadrangulation were of degree 2 or not.
If this is the case, we must replace at least one of its two incident edges with a quadrangulation of a 2-cycle,
obtaining the correcting factor $(2\tilde{A} + \tilde{A}^2)$, hence  the term $(2\tilde{A} + \tilde{A}^2)Q_1$.
If neither $u$ nor $v$ were of degree 2 we get $s(2N_0+N_1+R)$, where the factor $s$
{accounts for replacing the pair of edges that corresponds to the external face.}

As for the last equation, vertices of degree 2 in the root face can become isolated after replacing two consecutive edges of a quadrangle by quadrangulations of a 2-cycle.
There are four possibilities for this situation, encoded by $4z\widetilde{A}^2$ in the expansion of $z(1+\widetilde{A})^4$ (see Figure \ref{fig:quad_M0M1}).
Hence the term  $z(1+\widetilde{A})^4-4z\widetilde{A}^2+4zw\widetilde{A}^2$.
\end{proof}

We treat separately the generating function encoding the substitution of $P_3$, the path on 3 vertices, which corresponds to
\begin{equation}\label{eq:A_bar}
	 \widehat{A} = A_0 + A_1 + \frac{A_1}{w}.
\end{equation}
The difference with  $\widetilde{A}$ is that the two edges of $P_3$  have one endpoint of degree 1, and  when  the external 2-vertex of a quadrangulation  in $\mathcal{A}_1$  is identified with one of them,  its degree does not increase (see Figures \ref{fig:rootface} and \ref{fig:quad_M0'_M1}).

\begin{figure}[htb]
\centering
    	\includegraphics[scale=.7]{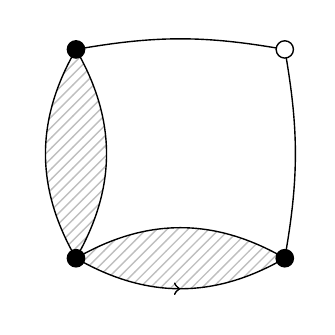}
    	\includegraphics[scale=.7]{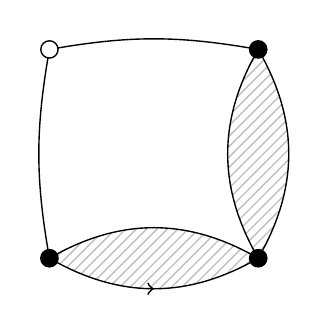}
    	\includegraphics[scale=.7]{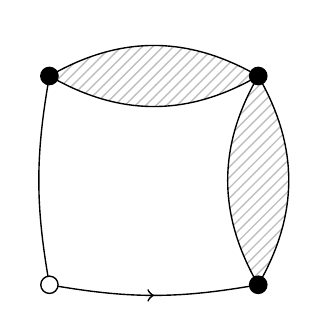}
    	\includegraphics[scale=.7]{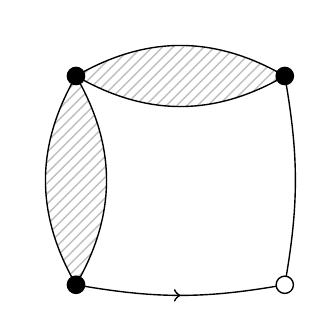}
	\caption{\em The four types of quadrangulations obtained by replacing two consecutive edges of a quadrangle with quadrangulations of a 2-cycle, where a 2-vertex (in white) is created. The  first two on the left are in $\cB_0$, while the two on the right are in $\cB_1$.}
	\label{fig:quad_M0M1}
\end{figure}

\begin{lemma}\label{lem:A}
	Let $Q_0,\,Q_1$ and $E$ be as in  Lemma~\ref{lemma:part}. Then  the following  equations hold and have a unique solution with non-negative coefficients:
	\begin{equation}\label{eqsA0}
		\renewcommand\arraystretch{1.5}
		\begin{array}{llcl}
			A_1 =& zw(1 + \widehat{A}),\\
			A_0 =& 2z\widetilde{A}(1 + \widehat{A}) + z(Q_0 + Q_1 + E + 2z\widetilde{A}(w - 1) + 2z\widetilde{A}^2(1 - w)).
		\end{array}
	\end{equation}
	Moreover, the system has a unique non-zero solution with non-negative coefficients.
\end{lemma}

\begin{proof}
	The right-hand terms are clearly divisible by $z$.
	The right-hand side of the second equation has non-negative coefficients.
	The term $-2z\widetilde{A}$ cancels with the first term, and $-2zw\widetilde{A}$ cancels with a corresponding term in $E$.
	Lemma \ref{positive} then guarantees the uniqueness of the solution with non-negative coefficients.

	We first observe that a quadrangulation of a 2-cycle can be thought as an ordinary quadrangulation adding an edge parallel to the root edge.
	
	When removing the root vertex of a quadrangulation of a 2-cycle in $\cA_1$, we obtain either an edge (term $zw$), a quadrangulation of a 2-cycle in $\cA_0$ (term $zwA_0$), a quadrangulation of a 2-cycle in $\cA_1$ (term $z A_1$), or its symmetric.
	In the last case, we create a 2-vertex, hence the factor $w$ in $zwA_1$.
	The reverse operation consists of starting from  a quadrangulation $\gamma$ encoded in $1+\widehat{A}$, adding a new vertex $v$ in the outer face, connecting $v$ to the root vertex $v'$ of $\gamma$ by two edges, and rooting the resulting map at $vv'$ so that the outer face is a digon.

	The term $2z\widetilde{A}(1 + \widehat{A})$ in the equation for $A_0$ encodes quadrangulations of the 2-cycle arising  from a quadrangulation whose external face is not a simple 4-cycle (illustrated in the two leftmost parts of Figure \ref{fig:rootface}  and on the right of Figure \ref{fig:quad_M0'_M1}).
	The terms $zQ_0+zQ_1+zE$ arise from the corresponding quadrangulations when building the 2-cycle.
	The term $zE$ has to be adjusted, because either we create a 2-vertex (term $2z^2\widetilde{A}(w-1)$, illustrated by the two graphs on the left-hand side of Figure \ref{fig:quad_M0'_M1}) or we remove a 2-vertex (term $2z^2\widetilde{A}^2(1 - w)$, illustrated by the two graphs on the right-hand side of Figure~\ref{fig:quad_M0M1}).
\end{proof}

From the previous lemma we can obtain the generating functions associated to $\cB_0$, $\cB^*_0$ and $\cB_1$.
\begin{lemma}\label{lem:M}
	Let $Q_0,Q_1,E,A_0,A_1$ be as in the previous two lemmas, and $\widehat{A}$ as defined  in Equation \eqref{eq:A_bar}.
	Then  $B_0$, $B_1$ and $B^*_0$ are given by
	\begin{equation}\label{eqsA}
		\renewcommand\arraystretch{1.5}
		\begin{array}{llcl}
	 		B_0 =& 2z(1 + \widehat{A})(1 + \widehat{A} - A_1) + z(Q_0 + E - 2zw\widetilde{A}^2  - 2z\widetilde{A}),\\
			B_1 =& 2z(1 + \widehat{A})A_1 + zw(Q_1 + 2z\widetilde{A}^2),\\
			B^*_0 =& 2z^2\widetilde{A},\\
		\end{array}
	\end{equation}
	where $z$ marks faces and $w$ marks $2$-vertices.
\end{lemma}

\begin{proof}
	Recall that an arbitrary quadrangulation is obtained by substituting edges by quadrangulations of a 2-cycle in a simple quadrangulation.
	The factor $z$ in all equations is used to encode the outer face, which has not been considered in the previous generating functions.

	The term  $2z(1 + \widehat{A})(1 + \widehat{A} - A_1)$ encodes quadrangulations obtained from $P_3$.
	Notice that quadrangulations counted by the term $2z(1 + \widehat{A})A_1$ are in $\cB_1$ and must be removed from $\cB_0$.
	The second term $z(Q_0 + E - 2zw\widetilde{A}^2  - 2z\tilde{A})$ encodes quadrangulations obtained from simple quadrangulations whose root face is a 4-cycle.
	In this situation, we have to remove from $zE$ the terms $2z^2\widetilde{A}$ and $2z^2w\widetilde{A}^2$, which contribute to $B_0^{*}$ and $B_1$, respectively (see the two leftmost maps in Figure \ref{fig:quad_M0'_M1} and the two rightmost maps in Figure \ref{fig:quad_M0M1}, respectively).
	Finally, there is an extra contribution to $B_1$ with the term $zwQ_1$.
\end{proof}

From the Systems \eqref{eqsQ} and \eqref{eqsA} we can compute, by iteration, the coefficients to any order of all the series involved.
In particular we can compute the coefficients of the series $B_1$, $B_0$ and $B_0^*$, which are needed in the next section.

\section{Rooted 3-connected 4-regular planar maps}\label{sec:3conn}

In this section we count  3-connected 4-regular planar maps according to the number of simple and  double edges.
Because they have a unique embedding on the oriented sphere, this is equivalent to counting labelled 3-connected 4-regular planar graphs.
We notice that a 3-connected 4-regular map cannot have triple edges, and double edges must be vertex disjoint.
In addition, all maps in this section are rooted.

For brevity, a face of degree 2 not adjacent to another face of degree 2 is called a \emph{2-face}.
We say that an edge is in a 2-face if it is one of its two boundary edges.
An edge is \emph{ordinary} if it is not in  the boundary of a 2-face.
Since the number of edges is even,  the number of ordinary edges is also even.
Maps are counted  according to two parameters: the number of 2-faces, marked by variable $w$, and half the number of ordinary edges, marked by $q$.
Setting $w=q$ one recovers the enumeration of 4-regular maps according to half the number of edges.
Observe that the dual of a quadrangulation with $\ell$ 2-vertices  is a 4-regular map with $\ell$ 2-faces.

We need to  define the replacement of edges {and faces of degree 2} by maps.
Let $M$ be a rooted map.
Consider a fixed orientation of the edges in $M$; since in a rooted map all vertices and edges are distinguishable we can define such an orientation unambiguously.

\emph{Replacement of edges.}
Let $e=uv$ be an edge of $M$ and $N$ a map whose root edge is simple.
The replacement of $e$ with $N$ is the map obtained by the following operation.
Subdivide $e$ twice transforming it in to the path $uu'v'v$, remove the edge $u'v'$, and identify $u'$ and $v'$ with the end vertices of the {(previously deleted)} root edge of $N$, respecting the orientations.
An example is shown in Figure \ref{fig:map_networks}, bottom right.

\emph{Replacement of {faces of degree 2}.}
Let $(e,e')$ be the endpoints of a {face $f$ of degree 2 in $M$}, and $N$ a map whose root {face $f'$ is a 2-face}.
The replacement of {$f$} with $N$ is the map obtained by identifying {$(e,e')$} with the {endpoints of $f'$ after having deleted the two edges of $f'$, while} preserving the orientation and the embedding. As exemplified in Figure \ref{fig:map_networks} (bottom), the face of degree 2 can be a 2-face or not.

Notice that if $M$ and $N$  above are 4-regular, then the maps obtained by replacement of edges {or faces of degree 2} are also 4-regular.

We now consider the following families of (rooted) 4-regular maps.
\begin{itemize}
	\item
		$\cM = \cM_0 \cup \cM^*_0\cup \cM_1 $ are 4-regular maps.
	  	$\cM_0 \cup \cM_0^*$ are 4-regular maps in which the root edge is not  incident with a 2-face, and $\cM_1$ are  those for which the root edge  is incident with exactly one 2-face.
	  	$\cM_0^*$ are  maps in which the root is one of the extreme edges of a triple edge, corresponding to dual maps of quadrangulations on the left of Figure \ref{fig:quad_M0'_M1}.
    	These classes  are in bijection with the classes $\cB_0$, $\cB^*_0$ and $\cB_1 $ from the previous section, and Lemma \ref{lem:M} gives access to the associated generating functions.
\end{itemize}

\noindent
The next classes are all subclasses of $\cM$.
Given a map $M$, we let $M^-$ be the map obtained by removing the root edge $st$.
In accordance with the terminology introduced in the next section, the \emph{poles} are the endpoints $s,t$ of the root edge.

\begin{itemize}
	\item
		$\cL$ are \emph{loop} maps:  the root-edge is a loop.

	\item
		$\cS = \cS_0 \cup \cS_1$ are \emph{series} maps: $M^-$ is connected and there is an edge in $M^-$ that separates the poles.
		As above, the index $i=0,1$ refers to the number of 2-faces incident with the root edge.

\item
		$\cF$ are maps $M$ such that the face to the right of the root-edge is a 2-face, and such that $M - \{s,t\}$ is connected; see Figure \ref{fig:map_networks}.

		\item $\overline{\cF}$ are maps $M$ such that the face to the left of
		the root-edge is a 2-face, and such that $M - \{s,t\}$ is connected. By
		symmetry, $\overline{\cF}$ is in bijection with $\cF$.
		
	\item
		$\cP = \cP_0 \cup \cP_1$ are \emph{parallel} maps:  $M^-$ is connected, there is no edge in $M^-$ separating the poles, and either {$M - \{s,t\}$ is disconnected or $st$ is a double or a triple edge of $M^-$.}
		The index $i=0,1$ has the same meaning as in the {series} class.
		However, the classes  $\cF, \overline{\cF}$ and $\cM_0^*$ are excluded from $\cP$.
	

	\item
		$\cH$ are  $h$-maps: they are obtained from a 3-connected 4-regular map $C$ (the {\em core}) with a root edge which is simple by possibly replacing every non-root ordinary edge of $C$ with a map in $\cM$, {as well as possibly replacing every 2-face (notice that in a 3-connected 4-regular map all faces of degree 2 are 2-faces).}

\end{itemize}
\noindent
We introduce the classes $\mathcal{F}$ and $\overline{\mathcal{F}}$  in order to properly deal with replacements of faces of degree 2. Also, by symmetry, the generating function for $\overline{\mathcal{F}}$ is the same  as for $\mathcal{F}$.

\begin{figure}[bht]
\centering
		\includegraphics[scale=1.3]{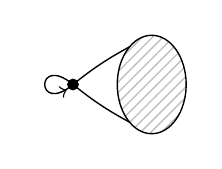}
		\includegraphics[scale=1.3]{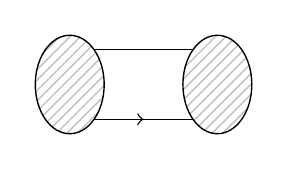}\\
		\includegraphics[scale=1]{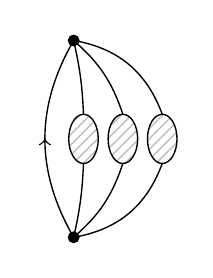}
		\includegraphics[scale=1]{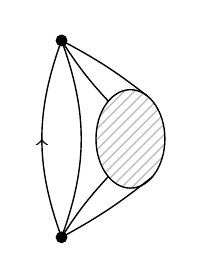}
		\includegraphics[scale=1]{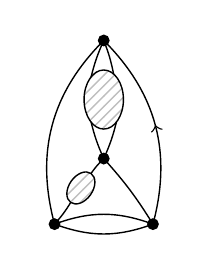}
		\includegraphics[scale=1]{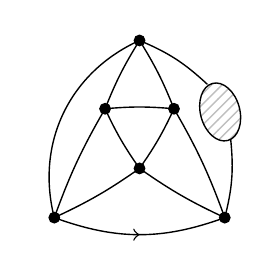}
\caption{Root-decomposition of 4-regular maps. On top a loop map (left) and a series map (right). Bottom, from left to right:  parallel map, map in $\cF$, and two maps in $\cH$.}
\label{fig:map_networks}
\end{figure}

\begin{lemma}\label{lem:partition_D}
	The former classes partition $\cM$ as follows
	$$
	\def\arraystretch{1.8}
	\begin{array}{ll}
	\cM= \cM_0 \cup \cM_1 \cup \cM_0^*, \\
	\cM_0 = \cL \cup \cS_0 \cup \cP_0 \cup \cH, 	\\
	\cM_1 = \cS_1 \cup \cP_1 \cup \cF \cup \overline{\cF}.	\\
	\end{array}
	$$
\end{lemma}

\begin{proof}
Let $G$ be a 4-regular map rooted at $e=st$, and suppose $e$ is  not a loop, so we are not in the class $\cL$. Consider the 2-connected component $C$ containing $e$.
By Tutte's theory of decomposition of 2-connected graphs into 3-connected components (see \cite{dissymetry} for a detailed exposition), $C$ is either a series, parallel or $h$-composition.
Series and $h$-compositions correspond, respectively, to classes $\cS$ and $\cH$.
Parallel compositions are either {in $\cP$, $\cF$ or in $\overline{\cF}$}; we need to distinguish between these three classes when replacing a double edge in a map.
\end{proof}

The next step is to determine the algebraic relations among the generating functions of the previous classes.
{We use variable $q$ to mark half the number of ordinary edges, and $w$ the number of 2-faces.}
{For each class of maps we have the corresponding generating function written with the same letter.}

{When dealing with the equations, we need to be careful when the root edge is in a 2-face.}
{In the case of the generating functions $M_0,\,M_1,\,M_0^{\ast},S_0,\,P_0$ and $H$, we count the total number of both ordinary edges and 2-faces}.
{In $S_1,\, P_1$  we count the number of 2-faces minus 1 (instead of the number of 2-faces), and the number of ordinary edges plus $2$.}
{Finally, in $F$, we count the number of 2-faces minus 1, and the number of ordinary edges}.
{This convention will be more transparent when obtaining equations relating the different generating functions.}
{Also note that, by symmetry, $F=\overline{F}$.}

We {also} introduce an auxiliary generating function $D$.
It is combinatorially equivalent to the generating function associated to the class $\cM$, but with the following modification.
When the root edge of {the map is in $\mathcal{M}_1$}, it is {\it  not} counted by the variable $w$; when it is in $\mathcal{M}_0$, it is counted by the variable $q$. {Finally, when the map is in $\mathcal{M}_0^{\ast}$, it is counted by the variable $w$.}
{The reason for the different treatment  is that $D$ will be used to encode replacements on edges}.
As a preparation for the proof of Lemma \ref{le:maps}, observe  that when replacing and edge by map in $\mathcal{M}_0^{\ast}$, we are creating a new 2-face, hence we use the variable $w$ instead of $q$ to encode the root edge.
{This is also the reason why we need to distinguish the family $\mathcal{M}_0^{\ast}$ from $\mathcal{M}_0$.}
{The definition of  $D$} is a technical device that simplifies the forthcoming equations {(see the proof of Lemma \ref{le:maps})}.
	
Let $T(u,v)$ be the generating function of 3-connected 4-regular maps in which the root edge is simple, where $u$ marks half the number of simple edges and $v$ marks the number of double edges (let us recall again that the number of edges in a 4-regular graph is even).
The next result provides a link between the known series $M_0$, $M_0^*$ and $ M_1$, and the series $T$ we wish to determine.

\begin{lemma}\label{le:maps}
	The following system of equations holds, where $q$ marks half the number of ordinary edges, and $w$ the number of  2-faces.
	\begin{equation}\label{eqsDmaps}
		\renewcommand\arraystretch{1.3}
		\begin{array}{lcl}
     		M_0 &=& S_0 + P_0 + L + H,\\
     		M_1 &=& \ds\frac{w}{q}(S_1 + P_1 + 2qF),\\
     		M^*_0 &=& 2q^2 D, \\
     		D &=& M_0 + \ds\frac{q}{w}M_1 + \frac{w}{q}M^*_0,\\
     		L &=& 2q(1+D-L) + L(w+q),\\
     		S_0 &=& D(D - S_0 - S_1) - \ds\frac{L^2}{2},\\
	 		S_1 &=& \ds\frac{L^2}{2},\\
     		P_0 &=& q^2(1 + D + D^2 + D^3) + 2qDF,\\
	 		P_1 &=& 2q^2D^2,\\
     		H &=& \ds\frac{T\left(q(1+D)^2, w + q(2D + D^2) + F\right)}{1+D}.
		\end{array}
	\end{equation}
\end{lemma}

\begin{proof} {We deduce each equation separately.}

{\emph{Equation for $M_0,\, M_1$}}. They follow from Lemma \ref{lem:partition_D}.
In  the equation for $M_1$ 
we have to add the factor $w/q$ in order to take care of how we mark the variable in the
$\cS_1$, $\cP_1$ and  $\mathcal{F}$. In particular we get the expression $2wF=\frac{w}{q}(2qF)$.

{\emph{Equation for $M_0^{\ast}$}. The two possibilities are illustrated at  the bottom of Figure \ref{fig:P1_M0*}, and they  lead directly to the claimed expression.}

{\emph{Equation for $D$}}. {The generating function $D$ is associated to the family $\mathcal{M}$.}
{The expression is a direct consequence of  how we mark the root edge between in the three subclasses of $\mathcal{M}$.}

{\emph{Equation for $L$}}. Consider now loop maps. The  double-loop is the map consisting  of a single vertex and two loops.
A loop map is obtained by possibly replacing the non-root loop of the double-loop with a map in $\cM$.
We have two possibilities illustrated in Figure~\ref{fig:loops_S1}, giving
$$
	L = 2q(1+D-L) + L(q+w).
$$

\begin{figure}[tbh]
\centering
    	\includegraphics[scale=1.5]{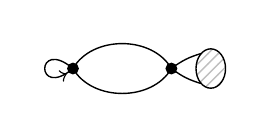}
    	\includegraphics[scale=1.5]{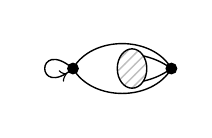}
    	\includegraphics[scale=1.5]{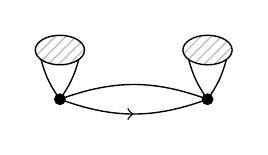}
\caption{\em On the left  the two cases of loop maps. On the right   is a  map in $\mathcal{S}_1$, in which  two loop maps are connected  in series, thus creating a 2-face.}
\label{fig:loops_S1}
\end{figure}
{\emph{Equation for $S_0,\,S_1$}}. A series map is obtained by taking a map in $\cM$ with poles $s_1$ and $t_1$, a map in $\cM \setminus \cS$ with poles $s_2$ and $t_2$, and then replacing  $s_1t_1$ and $s_2t_2$ with edges $t_1s_2$ and $s_1t_2$, the latter being the new root.
When the two maps connected in series are in $\cL$, a 2-face is created containing the root (see Figure \ref{fig:loops_S1}), and we obtain a map in~$\cS_1$.
There are four ways to connect two loop maps in series, but  only two of them produce a map in $\cS_1$:  when they are rooted either both on a face of degree one, or both on a face of degree more than one.
In terms of generating functions  we  have
$$
	S_0 = D(D-S) - S_1 \quad \text{and} \quad S_1 = \frac{L^2}{2}.
$$
{\emph{Equation for $P_0,\,P_1$}}. A parallel map can be obtained in two different ways.
First, from a map either in $\cF$ or  $\overline{\cF}$ with double root edge $r$ and replacing exactly one edge in $r$ with a map in $\cM$.
This is encoded by $2qFD$ (see Figure \ref{fig:P1_M0*}).
Secondly, we take  the 4-bond (an edge of multiplicity 4) and replace any of its three non-root edges with maps in $\cM$.
This is encoded by $q^2(1+D)^3$ (see Figure \ref{fig:map_networks}), but  two particular cases must be considered:

\begin{itemize}
\item[1)] When exactly two edges of the 4-bond are substituted, encoded by $P_1$ (see Figure \ref{fig:P1_M0*}), where a double edge is created;

\item[2)] When exactly one edge of the 4-bond is substituted.
\end{itemize}

Again there are two instances, encoded by  $M_0^*$ (see Figure \ref{fig:P1_M0*}), where the root edge belongs to exactly one of the two faces of degree 2. 	
Notice that these faces are not isolated and hence are encoded as three ordinary edges.
When the root edge is removed,
{the face of degree 2 containing it is removed and the remaining one becomes a 2-face.}
Summing this up gives the expressions for $P_0$ and $P_1$.

\begin{figure}[htb]
\centering
	\includegraphics[scale=1.2]{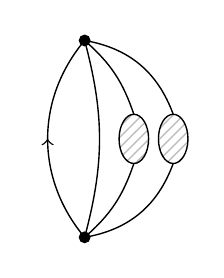}
	\includegraphics[scale=1.2]{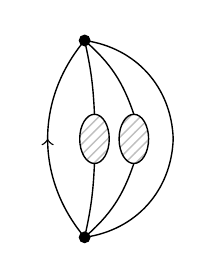}
	\includegraphics[scale=1.2]{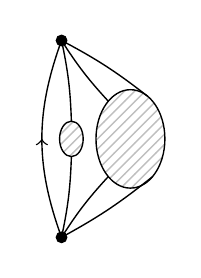}\\
	\includegraphics[scale=1.2]{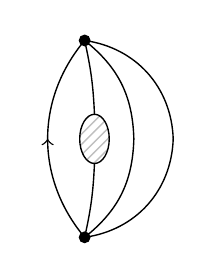}
	\includegraphics[scale=1.2]{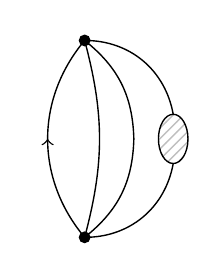}
\caption{\em Top from left to right:  two types of maps in $\cP_1$,
	and a map in $\cP_0$ obtained by taking a map in $\cF$ and substituting its root edge by a map in $\cM$. Bottom:  the two types of maps in $\cM_0^*$.}
\label{fig:P1_M0*}
\end{figure}

{\emph{Equation for $H$}}. A map in $\cH$ is obtained by possibly replacing the non-root simple edges of a core by a map in $\cM$, and the double edges with either:
\begin{itemize}
\item[1)] A map in $\cM$ on one of the edges of the double edge;
\item[2)] Two maps in $\cM$, one on each edge;
\item[3)] A map in $F$ (see Figure \ref{fig:map_networks}).
\end{itemize}

This gives
$$
	H = \frac{T\left( q(1 + D)^2, w + q(2D + D^2) + F \right)}{1 + D}.
$$
\end{proof}

\paragraph{Proof of Theorem \ref{th:3-conn}.}
From the knowledge of $M$ and the $M_i$ in the previous section we can determine~$D=M^*_0/(2q^2)$.
Hence we can also determine $L$, then $S_0$ and $S_1$.
We also know $P_1=2wD^2$ and from $M_1 = S_1 + P_1 + 2wF$ we obtain $F$.
This is enough to compute $P_0$, and from $M_0 = S_0+P_0+L+H$ we determine $H$.
Since $M$ and the $M_i$ were algebraic functions, so are all the functions in the previous system.

We are ready for the final step.
Consider the following change of variables
$$
	u = q(1+D)^2, \qquad v = w + q(2D + D^2) + F,
$$
relating $T$ and $H$.
The first terms in the  expansion of $u$ and $v$ in $q$ and $w$ are
$$
	u = q + \cdots, \qquad  v = w + \cdots
$$
It follows that the Jacobian at $(0,0)$ is equal to 1 and the system can be inverted, in the sense that we can determine uniquely the coefficients of the inverse series.
Computationally, this can be explicitly obtained using Gr\"obner basis (we are grateful to Manuel Kauers for this observation).

Let the inverse of the system be
$$
	q = a(u,v), \qquad  w = b(u,v).
$$
Since $D$ and $F$ are algebraic functions, so are the inverse functions $a$ and $b$.
Now we use the last equation in Lemma \ref{le:maps} to get
$$
	T(u,v) = \left(1+D(a(u,v),b(u,v))\right) H(a(u,v),b(u,v)).
$$
This equation determines $T$.
Since all the series involved are algebraic, so is $T$.

Recall that $t_n$ is the number of labelled 3-connected 4-regular planar graphs.
Let $T_n=[u^k] T(u,0)$ be the number of simple rooted 3-connected 4-regular maps.
Then we have the relation
$$
	8n t_n = n! T_n.
$$
This follows by double counting.
We can label the vertices of a rooted map in $n!$ different ways, since vertices in a rooted map are distinguishable, and on the other hand, from a labelled graph we obtain $8n$ rooted maps: $4n$ choices for the directed root edge, and 2 choices for the root face.
As a consequence,
$$
	8 u\tau'(u) = T(u,0).
$$
Since $T$ is algebraic, so is $\tau'(u)$.
\qed

\section{Labelled 4-regular planar graphs}\label{sec:graphs}

In this section we complete the proof of Theorem \ref{th:main}.
In the sequel all graphs are labelled.
First we define networks.
 A \emph{network} is a connected 4-regular multigraph $G$ with an ordered pair of adjacent vertices $(s, t)$, such that the graph obtained by removing the edge $st$ is simple.
Vertices $s$ and $t$ are the \emph{poles} of the network.

We now define several classes of networks, similar to the classes introduced in the previous section.
We use the same letters, but they now represent classes of labelled \emph{graphs} instead of classes of maps. No confusion should arise since in this section we deal  with graphs.

\begin{itemize}

\item $\cD$ is the class of all networks.

\item $\cL, \cS, \cP$ correspond as before to loop, series and parallel networks. We do not need to distinguish between $\cS_0$ and $\cS_1$ and between $\cP_0$ and $\cP_1$.

\item $\cF$ is the class of networks in which the root edge has multiplicity exactly 2 and removing the poles does not disconnect the graph.

\item $\cS_2$ are networks in $\cF$ such that after removing the two poles there is a cut vertex, see Figure \ref{fig:S2_H2}.

\item $\cH = \cH_1 \cup \cH_2$ are $h$-networks: in $\cH_1$ the root edge is simple and in $\cH_2$ it is double, see Figure \ref{fig:S2_H2}.
\end{itemize}

\begin{figure}
\centering
	\includegraphics[scale=1.3]{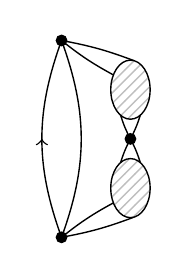} \hspace{2cm}
	\includegraphics[scale=1.3]{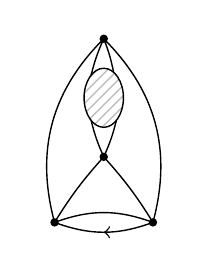}
\caption{\em The two types of networks in $\cF$. On the left  a network in $\cS_2$.
	On the  right  a network in $\cH_2$.}
\label{fig:S2_H2}
\end{figure}

\begin{lemma}\label{lem:partition_F}
	The two classes $\cS_2$ and $\cH_2$ partition $\cF$, that is
	$$
		\cF = \cS_2 \cup \cH_2.
	$$	
\end{lemma}

\begin{proof}
	Clearly  $\cS_2$ and $\cH_2$ are contained in $\cF$.
Now take a network  in $\cF$ and remove its two poles.
By definition, the resulting graph must be connected and has either a cut-vertex or is at least 2-connected.
If it has a cut-vertex then it belongs to $\cS_2$.
Otherwise, by the same argument as in Lemma \ref{lem:partition_D} it is obtained from a 3-connected core rooted at a double edge, hence it belongs to $\cH_2$.
\end{proof}

The  generating functions of networks are of the exponential type, for instance $D(x) = \sum D_n \frac{x^n}{n!}$.
We need in addition the generating functions
$$
	T_i(x,u,v) = \sum T^{(i)}_{n,k,\ell} u^k v^\ell \frac{x^n}{n!}
$$
of 3-connected 4-regular planar graphs rooted at a directed edge, where $i=1,2$ indicates the multiplicity of the root, $x$ marks vertices and $u,v$ mark, respectively, half the number of simple and the number of double edges.
The coefficients $T^{(i)}_{n,k,\ell}$ of  $T_i$ are easily obtained from those of $T=\sum t_{k,\ell} u^k v^\ell$, the generating function of 3-connected 4-regular maps computed in the previous section.
By double counting we have
$$
	T^{(1)}_{n,k,\ell} = n! \frac{t_{k,\ell}}{2},
	\qquad
	\ell \,T^{(2)}_{n,k,\ell} =k\,  T^{(1)}_{n,k,\ell}, \qquad n=\ell + k/2.
$$
Since we can compute the coefficients $t_{k,\ell}$ as in the previous section, we can compute the coefficients $T^{(i)}_{n,k,\ell}$ as well.

In terms of generating functions this amounts to
\begin{equation}\label{eq:T1_T2}
	T^{(1)}(x,u,v) =\frac{1}{2} T(u^2x,vx),
	\qquad
	u\frac{\partial}{\partial u}T^{(2)}(x,u,v) = v	\frac{\partial}{\partial u}T^{(1)}(x,u,v).
\end{equation}
Since $T$ is an algebraic function, $T^{(1)}$ is algebraic too, but to prove that $T^{(2)}$ is algebraic needs a separate argument.

Similarly to Lemma \ref{lem:partition_F} but in the maps setting, $\cF$ (as a class of maps) is partitioned into  $\cS_2$ and $\cH_2$.
Although redundant for the purpose of extracting the coefficients of $T$, this decomposition can be added to the system of equations \eqref{eqsDmaps} as follows:
\begin{equation}\label{eq:algebraicity_T2}
	\renewcommand\arraystretch{1.65}
	\begin{array}{lcl}
		F &=& S_2 + H_2,\\
		S_2 &=& (w + q(2D + D^2) + F)(w + q(2D + D^2) + F - S_2),\\
		H_2 &=& \ds\frac{T_2\left( q(1+D)^2, w + q(2D + D^2) + F \right)}{w + q(2D + D^2) + F}
	\end{array}
\end{equation}
where $T_2(u,v)$ counts 3-connected 4-regular \emph{maps} rooted at a double edge and with the face of degree two on the right of the root edge.
Since  $F$ and $D$ are algebraic by Lemma \ref{le:maps}, so are  $S_2$ and $H_2$.
And so is  $T_2$ since it can be derived form $H_2$ by the same algebraic inversion as before.
Finally we have
\begin{equation}
	T^{(2)}(x,u,v) = \frac{1}{2} T_2(u^2x,vx),
\end{equation}
where the division by two encodes the choice of the root face.

\begin{lemma}\label{lem:graphs}
	The following system of equations among the previous series holds:
	\begin{equation}\label{eq:Dgraphs}
 		\renewcommand\arraystretch{1.6}
		\begin{array}{lcl}
			D &=& L + S + P +H_1 + F\\
			L &=& \ds\frac{x}{2} (D-L) \\
			S &=& D(D-S)\\
			P &=& x^2\left(\frac{D^2}{2} + \frac{D^3}{6} \right) + FD \\
			F &=& S_2 + H_2 \\
			S_2 &=& \frac{1}{x}\left(F + x^2\left(D + \frac{D^2}{2}\right)\right)\left(F + x^2\left(D + \frac{D^2}{2}\right) - S_2\right) \\
			H_1 &=& \ds\frac{T^{(1)}\left(x,1+D, D+ \frac{D^2}{2} + \frac{F}{x^2}\right)}{1+D} \\
			H_2 &=& \ds\frac{T^{(2)}\left(x,1+D,D + \frac{D^2}{2} + \frac{F}{x^2}\right)}{D + \frac{D^2}{2} + \frac{F}{x^2}}
		\end{array}
\end{equation}
\end{lemma}

\begin{proof}
The first equation follows from a direct adaptation of Lemma~\ref{lem:partition_D} to the context of graphs.
The remaining equations follow by adapting the proof of Lemma~\ref{le:maps} from maps to graphs. We  briefly indicate the differences.

In the equation for $L$ we must take  into account that graphs are not embedded, hence the division by two, and also that the double loop is not admissible as a network.
In the equation for $S$ the only difference with  Lemma \ref{le:maps} is that the family $\cS_1$ is no longer needed.
Similar considerations apply to the equation for $P$. The equation for $F$ follows directly  from Lemma \ref{lem:partition_F}.

The  equation for $\cS_2$ describes the decomposition of a network in $\cS_2$.
It is essentially a series composition of two networks, one in $\cF$ and one in $\cF \setminus \cS_2$, in which the second pole of the first one is identified with the first pole of the second one, hence the division by $x$.
In addition we have  the cases were one of the two networks in the series composition is a {\em fat polygon}, that is, a cycle in which each edge is doubled, and every double edge is replaced with one or two networks in $\cD$.
Notice  that contrary to double edges of networks counted in $T^{(1)}$ and $T^{(2)}$, double edges of the fat polygon are not replaced with networks in $\cF$, as this would create a series composition with two networks in $\cS_2$.

The last two equations are similar to that for $H$ in Lemma \ref{le:maps} with the difference that double edges must be replaced with either: a network in $\cD$;  two networks in $\cD$, encoded by $D^2/2$; or a network in $\cF$ for which the two poles were removed, encoded by $F/x^2$.
\end{proof}

\paragraph{Proof of Theorem \ref{th:main}.}
Equations (\ref{eq:Dgraphs}) can be rewritten as a system with non-negative coefficients using the identities  $D-L=S+P+H_1+F$, $D-S=L+P+H_1+F$, and $F-S_2=H_2$.
Hence it has a unique solution with non-negative terms, which can be computed by iteration from the knowledge of $T^{(1)}$ and $T^{(2)}$.
Since all the functions involved are algebraic, the solution consists of algebraic functions.

Let $C(x)$ now be the generating function of labelled 4-regular planar graphs.
There is a simple relation between $C(x)$ and the series $D(x)$ of networks, namely
$$
	4xC'(x) = D(x) - L(x) - L(x)^2 - F(x)  - \frac{x^2}{2} D(x)^2.
$$
The series on the left corresponds to labelled graphs with a distinguished  vertex $v$ in which one of the 4 edges incident with $v$ is selected.
These correspond precisely to networks, except for the fact that since we are counting simple graphs we have to remove from $D(x)$ networks containing either loops or double edges, which correspond to the terms subtracted.

Finally, since $D,L$ and $F$ are algebraic functions, so is $C'(x)$.
\qed

\paragraph{Proof of Corollary \ref{coro}.}
A series is $D$-finite if it satisfies a linear differential equation with polynomial coefficients.
It is well-known (see Chapter 6 in \cite{stanley2}) that $\sum f_nx^n/n!$ is $D$-finite if and only if $\{f_n\}$ is $P$-recursive.
Since $G(x) = \exp(C(x))$ and $C'(x)$ is algebraic, it is enough to prove the following:

\begin{center}
	If $C'(x)$ is algebraic then $\exp(C(x))$ is $D$-finite.
\end{center}

\noindent
Let $G(x) = e^{C(x)}$.
It is easily proved by induction that  each derivative $C^{(i)}(x)$  is a rational function of $C'(x)$ and $x$.
Using this observation and applying induction we have that $G^{(i)} = R_i(C',x) G$, where  $R_i$ is a rational function in $C'$ and $x$, and we set $R_0=1$.
Since $C'$ is algebraic, $\mathbb{Q}(C', x)$ is finite dimensional over $\mathbb{Q}(x)$, say of dimension~$k$.
Hence there are rational functions $S_i(x)$ such that $\sum_{i=0}^k  S_i(x) R_i(C',x) =0$.
It follows that
$$
	S_0(x)G + S_1(x) G' + \cdots + S_k(x) G^{(k)} = 0.
$$
proving that $G$ is $D$-finite.

\section{Simple 4-regular maps}\label{sec:simple-maps}

The enumeration of \emph{simple} 4-regular maps is obtained by adapting the arguments in the previous section to maps instead of graphs.
We define the various classes of simple maps exactly as we did for networks, keeping the same notation.
The decomposition scheme starts from the series $T_1(x,u,v)$ and $T_2(x,u,v)$ of 3-connected 4-regular maps, where the indices and variables have the same meaning as in the previous section, with the exception that now the rooted (double) edge of $T_2$ does not need to have a face of degree two on its right hand-side.
As before, they are accessible from the series $T(u,v)$:
\begin{equation*}
	T_1(x,u,v) = T(xu^2,xv) \qquad \text{and} \qquad
	u\frac{\partial}{\partial u} T_2(x,u,v) = 2v\frac{\partial}{\partial v}T_1(x,u,v),
\end{equation*}
where  multiplication by 2 on the right-hand side of the second equation is because a double edge can be rooted at any of its two edges (as discussed above).

\begin{lemma}\label{lem:simple_maps}
	The following equations hold:
	\begin{equation}\label{eq:simple_maps}
		\renewcommand\arraystretch{1.8}
		\begin{array}{lcl}
			D &=& L + S + P + H_1 + 2F\\
			L &=& 2x(D - L) \\
			S &=& D(D - S)\\
			P &=& x^2(3D^2 + D^3) + 2FD \\
			F &=& S_2 + \ds\frac{H_2}{2}\\
			S_2 &=& \ds\frac{1}{x}(F + x^2(2D + D^2))(F + x^2(2D + D^2) - S_2)\\
			H_1 &=& \ds\frac{T_1\left(x, 1+D, 2D + D^2 + \frac{F}{x^2}\right)}{1 + D} \\
			H_2 &=& \ds\frac{T_2\left(x, 1+D, 2D + D^2 + \frac{F}{x^2}\right)}{2D + D^2 + \frac{F}{x^2}}
		\end{array}
	\end{equation}
\end{lemma}

\begin{proof}
The proof  is essentially the same as that of Lemma \ref{lem:graphs}, with the following differences.
Maps are embedded, hence there are $\binom{m}{k}$ ways to substitute $k$ maps in $\cD$ for an edge of multiplicity  $m$.
This  justifies the term $3D^2$ in the fourth equation and the term $2D$ in the last three equations.
Finally, the fact that maps have a root face explains the factors 2 in the first, second and fourth equations.
\end{proof}

\paragraph{Proof of Theorem \ref{th:maps}.}
From the knowledge of $T(u,v)$ of Section \ref{sec:3conn} and from Lemma \ref{lem:simple_maps}, we can compute the coefficients of the series $D$, $L$, $S$, $P$, $S_2$, $F$, $H_1$ and $H_2$ up to any order.
By removing maps having a loop or a multiple edge, the series $M(x)$ of rooted 4-regular simple maps is equal to
\begin{equation}
	M(x) = D(x) - L(x) - L(x)^2 - 3x^2D(x)^2 - 2F(x).
\end{equation}
Since $D,L$ and $F$ are algebraic, so is $M$. \qed


\section*{Acknowledgements}
We are grateful to Manuel Kauers and Bruno Salvy for useful suggestions concerning computational aspects of our work, and to Mireille Bousquet-M\'elou for suggesting the proof of Corollary \ref{coro}. 
We thank an anonymous referee for a careful reading of the manuscript and suggestions that helped to improve the presentation.
Part of the work presented in this paper was carried out during a visit of the second author at the Department of Mathematics of the Universitat Polit\`ecnica de Catalunya, Spain, and he would also like to thank this institution for welcoming him.

\newpage
\section*{Tables}

As an illustration we present in Table \ref{taula}  the numbers of 4-regular planar graphs up to 24 vertices.
We see that the first discrepancy is at $n=12$.
For this number of vertices there are 4-regular planar graphs which are either disconnected (the union of two octahedra) or are connected but not 3-connected (gluing two octahedra via two parallel edges).
Table \ref{tableT} and \ref{tab:simple_maps} give, respectively, the numbers of rooted 3-connected 4-regular maps, and simple 4-regular maps.

\begin{table}[htb]
\centering
\tiny
	\begin{tabular}{crrr}
	\toprule
		$n$ & $g_n$ & $c_n$ & $t_n$\\
	\midrule
		6 & 15 & 15 & 15 \\
		7 & 0 & 0 & 0 \\
		8 & 2520 & 2520 & 2520\\
		9 & 30240 & 30240 & 30240 \\
		10 & 1315440 & 1315440& 1315440\\
		11 & 39916800 & 39916800& 39916800 \\
		12 & 1606755150 & 1606651200 & 1546776000 \\
		13 & 66356690400 & 66356690400 & 63826963200 \\
		14 & 3068088823800 & 3067975310400 & 2879997120000 \\
		15 & 152398096250400 & 152395825982400 & 142057025510400\\
		16 & 8196374895508800 & 8196176020032000 & 7534165871232000\\
		17 & 472595587079616000 & 472586324386176000 & 430559631710208000\\
		18 & 29138462100216869400 & 29137847418231552000 & 26287924131076608000\\
		19 & 1912269800864459836800 & 1912231517504083776000 & 1710786280874711040000\\
		20 & 133143916957026288112800 & 133141260589657512192000 & 118162522829227548672000\\
		21 & 9803331490189678577136000 & 9803140616698955285760000 & 8635690901034837319680000\\
		22 & 761176404797020723326816000 & 761161832514030029322240000 & 665819208405772061921280000\\
		23 & 62162810722904469623293248000 & 62161644432203364801392640000 & 54014719048912416098304000000\\
		24 & 5327113727746428410913561441000 & 5327015666189741660374318080000 & 4599666299608288403199344640000\\
	\bottomrule
	\end{tabular}
	\caption{Numbers of arbitrary, connected and 3-connected labelled 4-regular planar graphs with $n$ vertices.}\label{taula}
\end{table}

\begin{table}[h]
\centering
\tiny
	\begin{tabular}{r|rrrrrrrrrrrrrrrrrrrrrrrr}
	\toprule
	$\ell \backslash k$ & 2 & 3 & 4 & 5 & 6 & 7 & 8 & 9 & 10 & 11 & 12 \\
	\midrule
	0 &&&&& 1 && 4 & 6 & 29 & 88 & 310\\
	1 &&&&& 12 & 28 & 128 & 396 & 1460 & 5148 & 18696\\
	2 & 2 & 6 & 16 & 40 & 156 & 546 & 2192 & 8316 & 32380 & 125510 & 489708\\
	3 && 8 & 56 & 260 & 1152 & 4900 & 21344 & 92160 & 397960 & 1708300 & 7303040\\
	4 &&& 46 & 510 & 3630 & 21350 & 115440 & 593622 & 2959160 & 14407250 & 68862960\\
	5 &&&& 312 & 4920 & 46508 & 347984 & 2282544 & 13791064 & 78760836 & 431601120\\
	6 &&&&& 2388 & 48860 & 579736 & 5267640 & 40819100 & 284712736 & 1843137520\\
	7 &&&&&& 19728 & 498352 & 7123464 & 76274560 & 683057672 & 5415222384\\
	8 &&&&&&& 172374 & 5190462 & 86891050 & 1072179834 & 10906813890\\
	9 &&&&&&&& 1571096 & 54988280 & 1055746780 & 14758457040\\
	10 &&&&&&&&& 14800940 & 590784084 & 12801068400\\
	11 &&&&&&&&&& 143190896 & 6422227344\\
	12 &&&&&&&&&&& 1415859276\\
	\bottomrule
	\end{tabular}
\caption{Coefficients of $T(u,v) = \sum t_{k,\ell} u^k v^\ell: t_{k,\ell}$ is the number of rooted 3-connected 4-regular maps with $\ell$ double edges and $2k$ simple edges, in which the root edge is simple.}
\label{tableT}
\end{table}

\begin{table}[htb]
\centering
\small
	\begin{tabular}{c|rr}
		\toprule
		$n$ & $t_{n,0}$ & $M_n$\\
		\midrule
		6 & 1 & 1\\
		7 & 0 & 0\\
		8 & 4 & 4\\
		9 & 6 & 6\\
		10 & 29 & 29\\
		11 & 88 & 88\\
		12 & 310 & 334\\
		13 & 1066 & 1196\\
		14 & 3700 & 4386\\
		15 & 13036 & 16066\\
		16 & 46092 & 59164\\
		17 & 164628 & 218824\\
		18 & 591259 & 812503\\
		19 & 2137690 & 3028600\\
		20 & 7770968 & 11329468\\
		21 & 28396346 & 42527120\\
		22 & 104256321 & 160148795\\
		23 & 384446150 & 604932614\\
		24 & 1423383358 & 2291617406\\
		\bottomrule
	\end{tabular}
\caption{$t_{n,0}$ is the number of simple rooted 3-connected 4-regular maps; $M_n$ is the number of simple 4-regular maps. As for graphs, the first discrepancy is at $n=12$. The numbers $t_{n,0}$ match those give in Table 1 from \cite{BDG93} given up to $n=15$.}\label{tab:simple_maps}
\end{table}

\newpage
\

\bibliographystyle{abbrv}
\newpage
\bibliography{biblio_4-regular}

\end{document}